\documentclass[11pt]{article}
\usepackage[a4paper]{anysize}\marginsize{1.5cm}{1.5cm}{1.5cm}{1cm}
\pdfpagewidth=\paperwidth \pdfpageheight=\paperheight
\usepackage{amsfonts,amssymb,amsthm,amsmath,eucal,tabu,url}
\usepackage{pgf}
 \usepackage{array}
 \usepackage{pstricks}
 \usepackage{pstricks-add}
 \usepackage{pgf,tikz}
 \usetikzlibrary{automata}
 \usetikzlibrary{arrows}
 \usepackage{indentfirst}
\theoremstyle{plain}
\newtheorem{thm}{Theorem}[section]
\newtheorem{theorem}[thm]{Theorem}

\newtheorem{proposition}[thm]{Proposition}

\theoremstyle{definition}
\newtheorem{definition}[thm]{Definition}

\newtheorem{example}[thm]{Example}

\newtheorem{problem}[thm]{Problem}

\newtheorem{thevarthm}[thm]{\varthmname}

\newenvironment{varthm*}[1]{\trivlist\item[]{\bf #1.}\it}{\endtrivlist}

\renewcommand\geq{\geqslant}

\renewcommand\leq{\leqslant}

\let\tilde=\widetilde

\newcommand\be{\begin{eqnarray*}}
\newcommand\ee{\end{eqnarray*}}

\newcommand\newop[2]{\def#1{\mathop{\rm #2}\nolimits}}
\newop\log{log}
\newop\ord{ord}
\newop\Gal{Gal}
\newop\SL{SL}
\newop\Bl{Bl}
\newop\mult{mult}
\newop\mass{mass}
\newop\div{div}
\newop\codim{codim}
\newop\sing{sing}
\newop\vdim{vdim}
\newop\edim{edim}
\newop\Ass{Ass}
\newop\size{size}
\newop\reg{reg}
\newop\satdeg{satdeg}
\newop\supp{supp}
\newop\Neg{Neg}
\newop\Nef{Nef}
\newop\Nefh{Nef_H}
\newop\Eff{Eff}
\newop\Zar{Zar}
\newop\MB{MB}
\newop\MBxC{MB\mathit{(x,C)}}
\newop\NnB{NnB}
\newop\Bigg{Big}
\newop\Effbar{\overline{\Eff}}

\def\keywordname{{\bfseries Keywords}}%
\def\keywords#1{\par\addvspace\medskipamount{\rightskip=0pt plus1cm
\def\and{\ifhmode\unskip\nobreak\fi\ $\cdot$
}\noindent\keywordname\enspace\ignorespaces#1\par}}
\def\subclassname{{\bfseries Mathematics Subject Classification
(2000)}\enspace}
\def\subclass#1{\par\addvspace\medskipamount{\rightskip=0pt plus1cm
\def\and{\ifhmode\unskip\nobreak\fi\ $\cdot$
}\noindent\subclassname\ignorespaces#1\par}}

\begin{document}
\title{Hirzebruch-type inequalities and plane curve configurations}
\author{Piotr Pokora}

\date{\today}
\maketitle
\thispagestyle{empty}
\begin{abstract}
In this paper we come back to a problem proposed by F. Hirzebruch in the 1980's, namely whether there exists a configuration of smooth conics in the complex projective plane such that the associated desingularization of the Kummer extension is a ball quotient. We extend our considerations to the so-called $d$-configurations of curves on the projective plane and we show that in most cases for a given configuration the associated desingularization of the Kummer extension is not a ball quotient. Moreover, we provide improved versions of Hirzebruch-type inequality for $d$-configurations. Finally, we show that the so-called characteristic numbers (or $\gamma$ numbers) for $d$-configurations are bounded from above by $8/3$. At the end of the paper we give some examples of surfaces constructed via Kummer extensions branched along conic configurations.

\keywords{curve configurations, Hirzebruch inequalities, ball-quotients}
\subclass{14C20, 52C35, 32S22}
\end{abstract}

\section{Introduction}
In this paper we come back to a question proposed by F. Hirzebruch. As we can read in a paper of I. Naruki \cite{Naruki}, Hirzebruch asked about interesting abelian covers of the complex projective plane branched along configurations of several conics. Naruki considered configurations of conics with nodes and tacnodes as singularities, unfortunately with no application into the direction of Hirzebruch's question. This topic was also studied later by Tang in \cite{Tang}, where the author constructed, in following to Hirzebruch's idea, abelian covers of the projective plane branched along configurations of smooth conics having pairwise tranvsersal intersection points. The main aim of Tang's paper was to find examples of conic confgurations such that the associated abelian cover (in fact the minimal desingularization) is a ball quotient, i.e., the universal cover of this surface is the unit ball. It turned out that Tang was not able to find such examples. In the present paper we extend Hirzebruch's question to a natural generalization of line and conic configurations, i.e., $d$-configurations, and we show that in most cases Hirzebruch's construction does not provide new examples of ball quotients -- there is one (combinatorial) family of curve configurations which potentially allows one to construct new ball quotients. In addition, we show two improvements of Hirzebruch-type inequalities obtained in \cite{PRSz, Tang} using results of Miyaoka \cite{M84} and Sakai \cite{Sakai}, and we show that the so-called characteristic numbers for $d$-configurations are bounded by $8/3$.

In the paper we work only over the complex numbers.
\section{Hirzebruch-type inequalities}
In his pioneering paper Hirzebruch \cite{Hirzebruch} constructed some new examples of algebraic surfaces which are ball quotients, i.e., surfaces of general type satisfying equality in the Bogomolov-Miyaoka-Yau inequality \cite{M84}
$$K_{X}^{2} \leq 3e(X),$$
where $K_{X}$ denotes the canonical divisor and $e(X)$ is the topological Euler characteristic. The key idea of Hirzebruch, which enabled constructing these new ball quotients, is that one can consider abelian covers of the complex projective plane branched along line configurations \cite{BHH87}. It turned out that Hirzebruch's construction can be performed for the so-called degree $d$-configurations with $d\geq 2$ -- see \cite{PRSz, Tang}.
 
\begin{definition}
Let $\mathcal{C} = \{C_{1}, ..., C_{\tau}\} \subset \mathbb{P}^{2}$ be a configuration of curves. Then $\mathcal{C}$ is a $d$-configuration of $\tau \geq 4$ curves if
\begin{itemize}
\item all irreducible components are smooth curves of degree $d \geq 1$,
\item all intersection points are transversal (i.e., pairwise intersections of curves are transversal),
\item there is no point where all curves meet.
\end{itemize}
\end{definition}
Let $\mathcal{C} = \{C_{1}, ..., C_{\tau} \} \subset \mathbb{P}^{2}$ be a $d$-configuration with $d \geq 2$. Now we can consider the Kummer extension of exponent $n \geq 2$ having degree $n^{\tau-1}$ and Galois group $(\mathbb{Z}/n\mathbb{Z})^{\tau-1}$ defined as the function field
$$K: = \mathbb{C}\left(z_{1}/z_{0}, z_{2}/z_{0}\right)\left((C_{2}/C_{1})^{1/n}, ...,(C_{\tau}/C_{1})^{1/n}\right).$$
This Kummer extension is an abelian extension of the function field of the complex projective plane.
It can be shown that $K$ determines an algebraic surface $\pi : X_{n} \rightarrow \mathbb{P}^{2}$ with normal singularities which ramifies over the plane with the configuration as the locus of the ramification. It can be shown, just as in the case of Hirzebruch's paper, that $X_{n}$ is singular over a point $p$ iff $p$ is a point of multiplicity $r_{p} \geq 3$ in $\mathcal{C}$. After blowing up these singular points we obtain a smooth surface $\rho : Y_{n}^{\mathcal{C}} \rightarrow X_{n}$. 

It turns out that the Chern numbers of $Y_{n}^{\mathcal{C}}$ can be read off directly from the combinatorics:
$$c_{2}(Y^{\mathcal{C}}_{n})/n^{\tau-3} = n^{2}(3 +(d^{2}-3d)\tau +f_{1}-f_{0})+n(-(d^{2}-3d)\tau -2f_{1}+2f_{0})+f_{1}-t_{2},$$
$$c_{1}^{2}(Y^{\mathcal{C}}_{n})/n^{\tau-3} = n^{2}( 9+d^{2}\tau -6d\tau  + 3f_{1} - 4f_{0}) +2n(-(d^{2}-3d)\tau -2f_{1}+2f_{0}) + 3d\tau + (d^{2}-3d)\tau +f_{1}-f_{0}+t_{2},$$
where $t_{r}$ denotes the number of $r$-fold points (i.e., points where exactly $r$ curves meet), $f_{0} = \sum_{r\geq 2}t_{r}$ and $f_{1} = \sum_{r\geq 2}r t_{r}$.
Moreover, it can be shown that $Y_{n}^{\mathcal{C}}$ has non-negative Kodaira dimension if $t_{\tau}= 0$ and $n \geq 2$ (more precisely, if $n\geq 3$, then all surfaces are of general type, and if $n=2$, then they are either elliptic or of general type) and in these cases we have $K_{Y_{n}^{\mathcal{C}}}^{2} \leq 3e(Y_{n}^{\mathcal{C}})$. Now we can define the following Hirzebruch polynomial:
\begin{equation}
\label{Hirzpoly}
H_{\mathcal{C}}(n) = \frac{3e(Y_{n}^{\mathcal{C}}) - K_{Y_{n}^{\mathcal{C}}}^{2}}{n^{\tau -3}} = n^{2}(f_{0} + (2d^{2}-3d)\tau) + n(-(d^{2}-3d)\tau -2f_{1} + 2f_{0}) + 2f_{1} + f_{0} -d^{2}\tau -4t_{2}
\end{equation}
and by the Bogomolov-Miyaoka-Yau inequality we have $H_{\mathcal{C}}(n) \geq 0$ (provided that $n \geq 2$). If there exists a configuration of curves $\mathcal{C}$ such that there exists $m \in \mathbb{Z}_{\geq 2}$ with $H_{\mathcal{C}}(m)=0$, then $Y_{m}^{\mathcal{C}}$ is a ball quotient. Let us introduce the following notion.
\begin{definition}
The surface $Y_{n}^{\mathcal{C}}$ obtained as the minimal desingularization of the Kummer extension of order $n^{\tau -1}$ branched along a given $d$-configuration $\mathcal{C}$ is called the \emph{Kummer cover}.
\end{definition}

If $X$ is a surface of general type and contains rational or elliptic curves, then the universal cover of $X$ cannot be the unit ball and hence $c_{1}^{2}(X) <3c_{2}(X)$. Therefore in this case one should be able to find a positive constant $m$ such that
$$3c_{2}(X) - c_{1}^{2}(X) \geq m.$$
It turns out that the constant $m$ can be explicitly computed under the assumption that the mentioned curves are smooth, which is provided by the results of Miyaoka \cite{M84} and Sakai \cite{Sakai}. Let us point out here that the result below has appeared for the first time in this form in \cite{Hirzebruch1}.
\begin{theorem}(Miyaoka-Sakai's improvement)
Let $X$ be a smooth surface of general type and $E_{1}, ..., E_{k}$ configurations (disjoint to each other) of rational curves (arising from quotient singularities) and $C_{1}, ..., C_{p}$ smooth elliptic curves (disjoint to each other and disjoint to the $E_{i}$'s). Let $c_{1}^{2}(X), c_{2}(X)$ be the Chern numbers of $X$. Then
$$3c_{2}(X) - c_{1}^{2}(X) \geq \sum_{j=1}^{p}(-C_{j}^{2}) + \sum_{i=1}^{k}m(E_{i}),$$
where $m(E_{i})$ is a constant which depends on the configuration.
\end{theorem} 

Let us come back here to Kummer extensions. Let $g: Z \rightarrow \mathbb{P}^{2}$ be the blow up of the projective plane along singular points of $\mathcal{C}$ with multiplicities $\geq 3$. It is known that there exists a morphism $\sigma: Y_{n}^{\mathcal{C}} \rightarrow Z$ such that $\rho \pi = g \sigma$. Let $q \in X_{n}$ be singular, then $C = \rho^{-1}(p)$ is a curve and $p = \pi(q)$ satisfies that the multiplicity $r_{p}$ of this point is $\geq 3$. Denote by $E_{p}$ the exceptional curve in $Z$ over $p$. Using the Hurzwitz's formulae we have
$$2-2g(C) =  n^{r_{p}-1}(2-r_{p}) + n^{r_{p}-2}r_{p}.$$
Moreover, observe that $\sigma^{*}E_{p}$ consists of $n^{\tau - r_{p}-1}$ disjoint curves $C$ (i.e., the so-called Fermat curves \cite[p.~28]{BHH87}), each of multiplicity $n$ and $C^{2} = -n^{r_{p}-2}$. We need the following general result for Kummer extensions which was formulated en passant in \cite[p.~140]{BHH87}.
\begin{proposition}
For $n=2$, $C$ is rational if and only if $r_{p} = 3$ and $C$ is elliptic if and only if $r_{p}=4$. If $n=3$, then $C$ is never rational and $C$ is elliptic if and only if $r_{p}=3$. If $n\geq 4$, then $C$ is neither elliptic nor rational.
\end{proposition}
In other words, if $n=2$ we know that $Y_{n}^{\mathcal{C}}$ has at least $t_{3}2^{\tau-4}$ rational curves with the self-intersection $-2$ and at least $t_{4}2^{\tau-5}$ elliptic curves of self-intersection $-4$. If $n=3$, then we have $t_{3}3^{\tau - 4}$ elliptic curves of self-intersection $-3$. Moreover, by \cite{Hemperly} we know that if $E_{i}$ consists of a single rational curve, then $m(E_{i}) = \frac{9}{2}$.
Now we are ready to show our first result.
\begin{theorem}
Let $\mathcal{C}$ be a $d$-configuration. Then
$$\left(\frac{7}{2}d^{2} - \frac{9}{2}d \right)\tau + t_{2} + \frac{3}{4}t_{3} \geq \sum_{r\geq 5}(r-4)t_{r}.$$
\end{theorem}
\begin{proof}
Considering the Hirzebruch polynomial with $n=3$ and Miyaoka-Sakai's improvement we have
$$3^{\tau-3}\left( (14d^{2} - 18d)\tau +4t_{2} + 4t_{3} -4 \sum_{r\geq 5}(r-4)t_{r} \right) \geq \sum_{j} (-C_{j}^{2}) = 3\cdot t_{3} 3^{\tau - 4},$$
which implies that
$$3^{\tau-3}\left( (14d^{2} - 18d)\tau +4t_{2} + 3t_{3} -4 \sum_{r\geq 5}(r-4)t_{r} \right) \geq 0.$$
After dividing by $4\cdot 3^{\tau-1}$ we get
$$\left(\frac{7}{2}d^{2} - \frac{9}{2}d \right)\tau + t_{2} + \frac{3}{4}t_{3} \geq \sum_{r\geq 5}(r-4)t_{r}$$
which completes the proof.
\end{proof}

Another Hirzbruch-type inequality has the following form.

\begin{theorem}
Let $\mathcal{C}$ be a $d$-configuration. Then
$$\left(5d^{2}-6d\right)\tau + t_{2} + \frac{3}{4}t_{3} \geq \sum_{r\geq 5}(2r-9)t_{r}.$$
\end{theorem}
\begin{proof}
Considering the Hirzebruch polynomial with $n=2$ and Miyaoka-Sakai's improvement we have
$$2^{\tau-3} \left( (5d^{2}-6d)\tau +t_{2} + 3t_{4} + t_{4} - \sum_{r\geq 5}(2r-9)t_{r} \right) \geq \frac{9}{2}\cdot t_{3}2^{\tau - 4} + 4\cdot t_{4}2^{\tau-5},$$
which implies that 
$$2^{\tau-3} \left( (5d^{2}-6d)\tau +t_{2} + \frac{3}{4}t_{3} - \sum_{r\geq 5}(2r-9)t_{r} \right) \geq 0.$$
After dividing by $2^{\tau -3}$ we obtain 
$$\left(5d^{2}-6d\right)\tau + t_{2} + \frac{3}{4}t_{3} \geq \sum_{r\geq 5}(2r-9)t_{r},$$
which completes the proof.
\end{proof}

\section{$d$-configurations and their characteristic numbers}

In this section we come back to Hirzebruch's idea of the so-called characteristic numbers of line configurations. Our aim is to study this object for arbitrary $d$-configurations. We defined \emph{ the characteristic number} of a given $d$-configuration $\mathcal{C}$ with $d\geq 1$ by
\begin{equation}
\label{gamma}
\gamma(\mathcal{C}) = \lim_{n\rightarrow \infty} \frac{c_{1}^{2}(Y_{n}^{\mathcal{C}})}{c_{2}(Y_{n}^{\mathcal{C}})} = \frac{9 + d^{2}\tau -6d\tau +3f_{1} -4f_{0}}{3 + (d^{2}-3d)\tau+f_{1}-f_{0}}.
\end{equation} 
Our main result is the following.
\begin{theorem}
Let $\mathcal{C}$ be a $d$-configuration. We may here assume additionally that for $d=1$ one has $t_{\tau-1}=0$ and $\tau \geq 6$. Then $\gamma(\mathcal{C}) \leq \frac{8}{3}$ and $\gamma(\mathcal{C}) = \frac{8}{3}$ if and only if $\mathcal{C}$ is the dual-Hesse configuration of lines.
\end{theorem}

\begin{proof}

One needs to combine \cite[Proposition II.8.]{Urzua} or \cite[Theorem 5.6]{Som} for $d =1$ and \cite[Theorem 3.6]{CP} for $d\geq 2$. For the completeness, let us present a detailed proof for $d\geq 2$.

We need to show that for a given $d$-configuration $\mathcal{C}$ we have $3 + (d^{2}-3d)\tau+f_{1}-f_{0} > 0$, which is obvious for $d\geq 3$, so let us consider $d=2$. By \cite{Tang} we know that for exponents $n\geq 3$ surfaces $Y_{n}^{\mathcal{C}}$ are of general type and it implies that $c_{2}(Y_{n}^{\mathcal{C}}) > 0$. Observe that $c_{2}(Y_{n}^{\mathcal{C}})$ is the quadratic polynomial with respect to $n$ and by $c_{2}(Y_{n}^{\mathcal{C}}) > 0$ we have $3 -2\tau +f_{1} - f_{0} \geq 0$. Now we show that the strict inequality also holds. To this end, we need to prove that the linear coefficient of $c_{2}(Y_{n}^{\mathcal{C}})$ is strictly negative, i.e., 
$$\tau - f_{1} + f_{0} <0.$$
We have the following inequalities:
\begin{itemize}
\item $2f_{0} \leq f_{1}$,
\item $\tau \leq f_{0}$, which is a consequence of \cite[Lemma 4.3]{PRSz}.
\end{itemize}
Thus we have
$$\tau + f_{0} \leq 2f_{0} \leq f_{1}.$$
Now, if $t_{2} < f_{0}$, then $2f_{0} < f_{1}$, and if $t_{2} = f_{0}$, then $t_{2} = 2(\tau^{2} - \tau) > \tau$, which completes the proof of our claim.

Suppose that $\gamma(\mathcal{C}) \geq 8/3$, then we have
$$3 + 6d\tau -5d^{2}\tau -t_{2} + \sum_{r\geq 5}(r-4)t_{r} \geq t_{2} + t_{3}.$$
Now using results from \cite{PRSz, Tang} we know that for a given $d$-configuration $\mathcal{C}$ one has
$H_{\mathcal{C}}(3) \geq 0$ and this provides us the following inequality
$$\bigg( \frac{7}{2}d^{2} - \frac{9}{2}d\bigg)\tau + t_{2} + t_{3} \geq \sum_{r\geq 5}(r-4)t_{r}.$$
Finally we have
$$3 + 6d\tau -5d^{2}\tau -t_{2} + \sum_{r\geq 5}(r-4)t_{r} \geq t_{2} + t_{3} \geq \bigg( \frac{9}{2}d - \frac{7}{2}d^{2}\bigg)\tau + \sum_{r\geq 5}(r-4)t_{r},$$
which gives
$$2t_{2} \leq 6 + 3d\tau(1 - d ) <0,$$
a
contradiction. 

\end{proof}

Let us observe that using $\gamma \leq 8/3$ for $d$-configurations we can obtain the following inequality.

\begin{proposition}
\label{f1}
Let $\mathcal{C}$ be a $d$-configuration. Then one has
$$3 + \sum_{r\geq 2}(r-4)t_{r} \leq (5d^{2} - 6d)\tau.$$
\end{proposition}
It is worth pointing out that for $d=1$ the above inequality also holds in the case $t_{\tau-1} =1$  \cite[Proposition II.8.]{Urzua}, so the assumption that one needs only $t_{\tau} = 0$ is here optimal.

\section{$d$-configurations and ball quotients}

In this section we check whether there exists a $d$-configuration $\mathcal{C}$ with $d\geq 2$ such that the associated Kummer cover $Y_{n}^{\mathcal{C}}$ is a ball quotient. It turns out that in most cases the answer is negative. In order to observe this phenomenon we will extensively use the theory of constantly branched covers which was developed in \cite{BHH87}. Kummer extensions are nice examples of such constantly branched covers, thus we can apply a general theory to our problem. Let us recall some facts from \cite[Section 1.3]{BHH87}. We know that if $Y^{\mathcal{C}}_{n}$ is a ball quotient, then all irreducible components of the ramification divisor $\sigma^{*}(\bar{\mathcal{C}})$, where $\bar{\mathcal{C}}$ is the total transform of $\mathcal{C}$ in $Z$, must satisfy $\text{prop}(E) = 2E^{2} - e(E) = 0$. In particular, each irreducible component $C$ of $\sigma^{*}E_{p}$ satisfies $C^{2} = -n^{r_{p}-2}$ and
$$\text{prop}(C) = n^{r_{p}-2}((r_{p}-2)(n-1)-4).$$
If $C_{j}$ is an irreducible component of $\mathcal{C}$, then we denote by $\tilde{C}_{j} = \sigma^{*}(C_{j}^{'})$, where $C_{j}^{'}$ is the strict transform of $C_{j}$ under blowing-up $g$. Then we have
$$\text{prop}(\tilde{C}_{j}) = n^{\tau-3}( \text{prop}(C_{j}) + (n-1)(r_{j} - e(C_{j})) - 2\delta_{j}),$$
where $r_{j}$ denotes the total number of singular points on $C_{j}$, $\delta_{j}$ denotes the number of essential singular points, i.e., those with multiplicity $\geq3$, and $r_{j,2}$ denotes the number of double points. This leads to $r_{j} = \delta_{j} + r_{j,2}$. If $C_{j}$ is a smooth plane curve of degree $d\geq 1$, then we have $\text{prop}(C_{j}) = 3d^{2} - 3d$, $e(C_{j}) = 3d - d^{2}$, and finally for each $\tilde{C}_{j}$ one has
$$\text{prop}(\tilde{C}_{j}) = n^{\tau-3}(3d^{2}-3d + (n-1)(r_{j} + d^{2}-3d) - 2\delta_{j}).$$
Now, the condition $\text{prop}(C) = 0$ leads to
$$(n-1)(r_{p}-2) = 4,$$ and the following pairs are admissible (we have the following order of listing: $(n,r_{p})$):
$$(5,3), \quad (3,4), \quad (2,6).$$
It means that $Y_{n}^{\mathcal{C}}$ can be a ball quotient if one of the following conditions is satisfied:
\begin{itemize}
\item $t_{r} = 0$ for $r \neq 2, 3$ and $n = 5$,
\item $t_{r} = 0$ for $r \neq 2, 4$ and $n = 3$,
\item $t_{r} = 0$ for $r \neq 2, 6$ and $n = 2$.
\end{itemize}
Now we need to use $\text{prop}(\tilde{C}_{j}) = 0$, i.e.,
$$3d^{2} - 3d + (n-1)(r_{j} +d^{2} - 3d) = 2\delta_{j}.$$
Simple computations lead to the following cases:
\begin{itemize}
\item $n = 5, r_{p} = 3: \quad 4r_{j,2} + 2r_{j,3} = 15d - 7d^{2}$, and if $d\geq 3$, then there is no $d$-configuration such that the associated Kummer cover is a ball quotient. We will consider $d=2$ separately.
\item $n = 3, r_{p} = 4: \quad 2r_{j,2} = 9d - 5d^{2}$, and if $d\geq 2$, then there is no $d$-configuration such that the associated Kummer cover is a ball quotient.
\item $n = 2, r_{p} = 6: \quad r_{j,2} + 4d^{2}-6d = r_{j,6} .$
\end{itemize}
Let us come back to $\text{prop}(\tilde{C}_{j}) = 0$ once again using the above list. We have previously observed that
$$3d^{2}-3d + (n-1)(r_{j,2} + r_{j,r} + d^{2}-3d) = 2r_{j,r},$$
and using $(n-1) = \frac{4}{r-2}$ one has
$$(3d^{3} - 3d)(r-2) + 4(r_{j,2} + r_{j,r} + d^{2} - 3d) = 2(r-2)r_{j,r}.$$
Then
$$(2r-8)r_{j,r} - 4r_{j,2} -4d^{2} + 12d + (2-r)(3d^{2} - 3d)=0.$$
Now we can use the fact that $(r-1)r_{j,r} + r_{j,2} = d^{2} (\tau - 1)$ and we obtain
$$r_{j,2} + r_{j,r} = \frac{2d^{2}(\tau-1) + 12d - 4d^{2} + (2-r)(3d^{2}-3d)}{6}.$$
Consider the case $d=2$, $n=5$, and $r = r_{p}=3$. We have
$$r_{j,2} + r_{j,3} = \frac{8\tau - 6}{6},$$
and using $\sum_{j} r_{j,r} = rt_{r}$ one finally has
$$2t_{2} + 3t_{3} = f_{1} =  \frac{\tau(4\tau - 3)}{3}.$$
On the other hand, for conic configurations there is the following combinatorial equality
$$2(\tau^{2} - \tau ) = t_{2} + 3t_{3}.$$
Combining things one obtains
$$t_{2} = \frac{3\tau - 2\tau^{2}}{3} < 0,$$
a contradiction.

Thus there does not exist a conic configuration with double and triple points such that the associated Kummer cover of order $5^{\tau-1}$ is a ball quotient. 

Now we consider the last remaining case, i.e., $n=2$ and $r = r_{p} = 6$. Some tedious computations lead to 
$$r_{j,2} + r_{j,6} = \frac{d^{2}(\tau-1) -8d^{2} + 12d}{3}$$ and 
$$2t_{2} + 6t_{6} = f_{1} = \tau \frac{d^{2}(\tau-1) -8d^{2} + 12d}{3}.$$
Using $2t_{2} + 30t_{6} = d^{2}(\tau^{2}-\tau)$ one obtains
\begin{equation}
\label{condition}
t_{6} = \frac{d\tau(d\tau + 3d - 6)}{36}, \quad t_{2} = \frac{d\tau(d\tau-21d+30)}{12}.
\end{equation}
It means that if we can find integers $d \geq 2, \tau \geq 4$ such that $t_{2}$ and $t_{6}$ are integers and configuration, denoted  by $\mathcal{C}_{d,\tau}$, is geometrically realizable over the complex numbers, then the associated Kummer cover $Y_{2}^{\mathcal{C}_{d,\tau}}$ might be a ball quotient. For instance, if $d=2$, then taking $\tau = 3m$ with $m\geq 3$ one has 
$$t_{6} = m^{2}, \quad t_{2} = 3m(m-2).$$
However, easy computations reveal that 
$$H_{\mathcal{C}_{2,3m}}(2) \neq 0.$$
\begin{theorem}
There does not exist a conic configuration $\mathcal{C}$ such that $Y_{n}^{\mathcal{C}}$ is a ball quotient.
\end{theorem}

This phenomenon can be explained in a slightly different way: for conic configurations $H_{\mathcal{C}}(n)$ cannot be written as a polynomial relation using $\text{prop}(\tilde{C}_{j})$ and $\text{prop}(\sigma^{*}E_{p})$ -- this can be done for line configurations.

Now we are ready to formulate another main result of this paper.
\begin{theorem}
Let $\mathcal{C}$ be a $d$-configuration with $d\geq 3$ and assume that $\mathcal{C}$ is not projective equivalent to some $\mathcal{C}_{d,\tau}$. Then the associated Kummer cover $Y_{n}^{\mathcal{C}}$ is never a ball quotient.
\end{theorem}

Let us now focus on some examples of surfaces constructed via Hirzebruch's method.

\begin{example}
In \cite{AD} the authors constructed the Hesse configuration of smooth conics $\mathcal{AD}$ consisting of $12$ conics, $12$ double points, and $9$ points of multiplicity $8$. Let us now consider the Kummer cover $Y_{n}^{\mathcal{AD}}$ of order $n^{11}$ branched along $\mathcal{AD}$. We can compute easily the Chern numbers, namely
$$c_{2}(Y_{n}^{\mathcal{AD}}) = 54n^{11} - 126n^{10} + 84n^{9},$$
$$c_{1}^{2}(Y_{n}^{\mathcal{AD}}) = 117n^{11} -252n^{10} + 135n^{9}.$$
Let us denote by $E(Y_{n}^{\mathcal{AD}})$ the Chern slope of $Y_{n}^{\mathcal{AD}}$, i.e., $E(Y_{n}^{\mathcal{AD}}) = \frac{c_{1}^{2}(Y_{n}^{\mathcal{AD}})}{c_{2}(Y_{n}^{\mathcal{AD}})}$. We have the following 
\begin{center}
\begin{tabular}{r|c|c|c|c|c} 
$n$ &  2 & 3 & 4 & 5 & $\gamma(\mathcal{AD})$ \\ \hline
$E(Y_{n}^{\mathcal{AD}})$ & 2.0625 & 2.25 & 2.25 & 2.2388 & 2.16 
\end{tabular}.
\end{center}
\end{example}
\begin{example}
We assume that $\mathcal{G}_{\tau}$ is a configuration of $\tau \geq 4$ general conics. It means that $t_{2} = 2(\tau^{2} - \tau)$ and $t_{r} = 0$ for $r>2$. Easy computations show that (we use the change of coordinates $n \mapsto x+1$):
$$H_{\mathcal{G}_{k}}(x) = 2\tau x(\tau x + 3) > 0$$ for every $x\geq 1$.
Now we would like to look at values of the Chern numbers. We have
$$c_{1}^{2}(Y_{n}^{\mathcal{G}_{k}})/n^{\tau-3} = (4\tau^{2} - 12\tau + 9)n^{2} + 4n(3\tau - 2\tau^{2}) + 4\tau^{2},$$
$$c_{2}(Y_{n}^{\mathcal{G}_{k}})/n^{\tau-3} = (2\tau^{2} - 4\tau + 3)n^{2} + 2n(3\tau - 2\tau^{2}) + 2\tau^{2} - 2\tau .$$
In particular, for $\tau = 4$ we have
$$c_{1}^{2}(Y_{n}^{\mathcal{G}_{4}}) = 25n^{3} - 80n^2 + 64n,$$
$$c_{2}(Y_{n}^{\mathcal{G}_{4}}) = 19n^3 - 40n^2 + 24n$$
and for $n=2$ one has $c_{1}^{2}(Y_{2}^{\mathcal{G}_{4}}) = 8$, $c_{2}(Y_{2}^{\mathcal{G}_{4}}) = 40$, and finally
$$E(Y_{2}^{\mathcal{G}_{4}}) = \frac{1}{5}.$$
On the other hand,
$$\lim_{\tau, n \rightarrow \infty} E(Y_{n}^{\mathcal{G}_{\tau}}) = 2.$$
It is worth pointing out that in this case our surface $X_{n}$ is smooth since there are only double points.
\end{example}
At the end of the paper, let us also recall a quite challenging problem which appears in \cite[p.~116]{BHH87}. This question is strictly related to our problem about configurations $\mathcal{C}_{d,\tau}$.  The authors have proposed a sequence of (combinatorial) line configurations $\mathcal{L}_{m}$ for $m\geq 3$ with $\tau = 12m+3$, $t_{2} = 12m^{2} + 15m+3$, and $t_{6} = 4m^{2} + m$. It can be easily checked that $H_{\mathcal{L}_{m}}(2) = 0$ for every $m\geq 3$. However, it is not known whether $\mathcal{L}_{m}$'s are realizable geometrically with straight lines. 

In a very recent paper, Shnurnikov \cite{Shnurnikov} has shown that if $\mathcal{L}$ is a line configuration realizable over the real numbers with $\tau \geq 6$ and $t_{\tau} = t_{\tau-1} = t_{\tau - 2} = t_{\tau -3} = 0$, then
$$t_{2} + \frac{3}{2}t_{3} \geq 8 + \sum_{r\geq 4}(2r-7.5)t_{r}.$$
Again, it is easy to check that $\mathcal{L}_{m}$'s do not satisfy Shnurnikov's inequality, thus they cannot be realized over the real numbers. This leads to the following problem.
\begin{problem}
Is it possible to realize $\mathcal{L}_{m}$ with $m \geq 3$ over the complex numbers?
\end{problem}
\section*{Acknowledgements}
I would like to express my gratitude to Professor Igor Dolgachev for pointing out Naruki's paper \cite{Naruki} and stimulating conversations on the topic of this paper, to Roberto Laface for conversations around the project, and to Xavier Roulleau for very useful remarks which allowed to improve this paper. The idea behind this paper grew up during the MFO Workshop \emph{Arrangements of Subvarieties, and their Applications in Algebraic Geometry} in 2016. I would especially like to thank Giancarlo Urz\'ua and Hal Schenck for stimulating conversations during this event. At last, I would like to thank the anonymous referee for useful suggestions. The author is partially supported by National Science Centre Poland Grant 2014/15/N/ST1/02102. 


\bigskip
   Piotr Pokora,
   Instytut Matematyki,
   Pedagogical University of Cracow,
   Podchor\c a\.zych 2,
   PL-30-084 Krak\'ow, Poland.

Current Address:
    Institut f\"ur Algebraische Geometrie,
    Leibniz Universit\"at Hannover,
    Welfengarten 1,
    D-30167 Hannover, Germany. \\
\nopagebreak
   \textit{E-mail address:} \texttt{piotrpkr@gmail.com, pokora@math.uni-hannover.de}


\end{document}